\documentclass[12pt]{amsart}

\usepackage[latin1]{inputenc}
\usepackage[english]{babel}
\usepackage{amsthm}
\usepackage{amsmath}
\usepackage{amsxtra}
\usepackage{mathrsfs}
\usepackage{amssymb}
\usepackage[dvips]{graphicx}
\usepackage{graphicx}

\input xy
\xyoption{all}

\newcommand{\Coll}{{\rm Coll}}

\newcommand{\lev}{{\rm Lev}}

\newcommand{\TP}{{\rm TP }}
\newcommand{\ITP}{{\rm ITP }}
\newcommand{\dom}{{\rm dom}}

\newcommand{\ZFC}{{\rm ZFC }}

\newcommand{\force}{\Vdash}
\newcommand{\spazio}{\textrm{ }}
\newcommand{\restr}{\upharpoonright}

\newtheorem{theorem}{Theorem}[section]
\newtheorem{lemma}[theorem]{Lemma}
\newtheorem{proposition}[theorem]{Proposition}
\newtheorem{coroll}[theorem]{Corollary}

\newtheorem{definition}[theorem]{Definition}
\newtheorem{notation}[theorem]{Notation}
\newtheorem{claim}[theorem]{Claim}
\begin{document}
\title{The Strong Tree Property at Successors of Singular Cardinals}


\author[Laura Fontanella ]{Laura Fontanella}
\urladdr{http://www.logique.jussieu.fr/$\sim$fontanella}
\address{ Equipe de Logique Math\'ematique,
Universit\'e Paris Diderot Paris 7, UFR de math\'ematiques case 7012, 
site Chevaleret, 75205 Paris Cedex 13, France}

\address{Kurt G\"{o}del Research Center for Mathematical Logic, 
University of Vienna, Department of Mathematics\\
W\"{a}hringer Strasse $25,$ Vienna $1090$ (Austria) 
}

\email{fontanl6@univie.ac.at}




\date{9 february 2012}



\maketitle

\begin{abstract}{} An inaccessible cardinal is strongly compact if, and only if, it satisfies the strong tree property. We prove that if there is a model of $\ZFC$ with infinitely many supercompact cardinals, then there is a model of \ZFC where $\aleph_{\omega+1}$ has the strong tree property. Moreover, we prove that every successor of a singular limit of strongly compact cardinals has the strong tree property. 
\end{abstract}

\




\section{Introduction}

The strong tree property is a strong generalization of the usual tree property. Given a regular cardinal $\kappa,$ we say that $\kappa$ has the tree property when every $\kappa$-tree (i.e. every tree of height $\kappa$ with levels of size less than $\kappa$) has a branch of length $\kappa.$
K\"{o}nig's Lemma establishes that the tree property holds at $\aleph_0.$ On the other hand, $\aleph_1$ does not satisfy the tree property, and for larger regular cardinals whether or not they satisfy the tree property is independent from $\ZFC.$ It is well known that the tree property provide a combinatorial characterization of weak compactness. 

\medskip

\noindent {\bf Theorem:} (Erd\"{o}s and Tarski \cite{ErdosTarski} $1961$) Assume $\kappa$ is an inaccessible cardinal, then $\kappa$ is weakly compact if and only if it satisfies the tree property.

\medskip

Strongly compact and supercompact cardinals admit similar characterizations. 

\medskip

\noindent {\bf Theorem:} If $\kappa$ is an inaccessible cardinal, then 
\begin{enumerate}
\item $\kappa$ is strongly compact if and only if it satisfies the strong tree property (Jech \cite{Jech} $1973,$ Di Prisco - Zwicker \cite{DiPriscoZwicker} $1980$ and Donder - Weiss \cite{WeissPhd} $2010$); 
\item $\kappa$ is supercompact if and only if it satisfies the super tree property (Jech \cite{Jech} $1973,$ Magidor \cite{Magidor} $1974$ and Donder - Weiss \cite{WeissPhd} $2010$).\\
\end{enumerate}  

\medskip

\noindent The strong and super tree properties generalize the usual tree property to the combinatorics of 
$[\lambda]^{<\kappa},$ in fact they concern special structures known as \emph{$(\kappa,\lambda)$-trees} that can be seen as ``trees over $[\lambda]^{<\kappa}$" whose ``levels'' have size less than $\kappa$ (this notion will be defined in \S \ref{sec:maindef}). 
The super tree property implies the strong tree property, that entails the usual tree property in its turn.
While the previous characterizations date back to the early $1970$s, a systematic study of the strong and the super tree properties has only recently been undertaken by 
Weiss\footnote{In Weiss' terminology, the strong tree property at a regular cardinal $\kappa$ corresponds to the 
property $(\kappa, \lambda)$-TP for all $\lambda\geq \kappa,$ while the super tree property corresponds to $(\kappa, \lambda)$-ITP for all $\lambda\geq \kappa.$} who worked on theese properties in his 
Ph.D thesis \cite{WeissPhd} and proved that even small cardinals can consistently satisfy the strong and the super tree properties, if we assume large cardinals.\\ 

There is a huge literature concerning the construction of models of set theory in which several distinct regular cardinals satisfy the usual tree property. We list a few classical results of that sort. 

\begin{enumerate}
\item (Mitchell \cite{Mitchell72} $1972$) Let $\tau$ be a regular cardinal such that $\tau^{<\tau}= \tau.$ Assume there is a model of ZFC with a weakly compact cardinal, then there is a model of ZFC where $\tau^{++}$ has the tree property.
\item (Cummings and Foreman \cite{CummingsForeman} $1998$) Assume there is a model of ZFC with infinitely many supercompact cardinals, then there is a model of ZFC where every cardinal of the form $\aleph_n$ with $2\leq n<\omega$ has the tree property. 
\item (Magidor and Shelah \cite{MagidorShelah} $1996$) Assume there is a model of ZFC with an increasing sequence $\langle \lambda_n\rangle_{n<\omega}$ such that 
\begin{enumerate}
\item if $\lambda= \sup_{n\geq 0} \lambda_n,$ then $\lambda_n$ is $\lambda^+$-supercompact, for all $n>0;$
\item $\lambda_0$ is the critical point of an embedding $j: V\to M$ where $j(\lambda_0)= \lambda_1$ and ${}^{{\lambda}^+}M\subseteq M.$
\end{enumerate}
Then there is a model of ZFC where $\aleph_{\omega+1}$ has the tree property.
\item (Sinapova \cite{Sinapova} $2012$) Assume there is a model of ZFC with infinitely many supercompact cardinals, then there is a model of ZFC where $\aleph_{\omega+1}$ has the tree property.
\item (Neeman \cite{Neeman} $2012$) Assume there is a model of ZFC with infinitely many supercompact cardinals, then there is a model of ZFC where the tree property holds at every $\aleph_n$ with $n\geq 2$ and at $\aleph_{\omega+1}.$ 
\end{enumerate} 

All these results were oriented toward the construction of a model where the tree property holds simultaneously at every regular cardinal  --- whether such a model can be found is still an open question.  Some of these theorems can be generalized to the strong or the super tree property. In fact, Weiss proved that for every integer $n\geq 2,$ if we force with Mitchell's forcing over a supercompact cardinal, we get a model of set theory where even the super tree property holds at $\aleph_n.$ The author \cite{Fontanella, Fontanella2} (and independently Unger \cite{Unger}) proved that Weiss result can be generalized to get a model where all cardinals of the form 
$\aleph_n$ with $2\leq n<\omega$ simultaneously satisfy the super tree property, starting from infinitely many supercompact cardinals. Indeed, a forcing construction by Cummings and Foreman produces a model where all the $\aleph_n$'s satisfy the super tree property. We are going to prove from large cardinals that 
even $\aleph_{\omega+1}$ can consistently satisfy the strong tree property. More precisely we will prove the following theorem. 

\medskip

{\bf Theorem:} If there is a model of ZFC with infinitely many supercompact cardinals, then there is a model of ZFC where $\aleph_{\omega+1}$ has the strong tree property. 

\medskip

\noindent The proof of such theorem is motivated by Neeman's paper \cite{Neeman}.  
By generalizing a theorem by Magidor and Shelah \cite{MagidorShelah}, we will also prove the following result. 

\medskip

{\bf Theorem:} If $\nu$ is a singular limit of strongly compact cardinals, then the strong tree property 
holds at $\nu^+.$

\medskip

Moreover, we will weakened the hypothesis of the latter theorem by using a partition property satisfied by strongly compact cardinals. 

\section{Preliminaries and Notation}

It may be useful to recall some terminology. The main reference for basic set theory is \cite{Jech}, while we will refer to \cite{Kanamori} for large cardinals notions and to \cite{Kunen} for the forcing technique. We denote by 
$[A]^{<\kappa}$ the set of all subsets of $A$ of size less than $\kappa.$
We recall the definition of closed unbounded subset of $[A]^{<\kappa}$ (club) and stationary subset of 
$[A]^{<\kappa}.$ 

\begin{definition} Assume $\kappa$ is a cardinal, $A$ is a set of size $\geq \kappa$ and $C\subseteq [A]^{<\kappa}.$
\begin{enumerate}
\item $C$ is \emph{unbounded}\index{unbounded} if for every $x\in [A]^{<\kappa}$ there exists $y\in C$ such that $x\subseteq y.$
\item $C$ is \emph{closed}\index{closed} if for any $\subseteq$-increasing chain 
$\langle x_{\gamma}\rangle_{\gamma<\alpha}$ of sets in $C,$ with $\alpha<\kappa,$ the union $\bigcup_{\gamma<\alpha} x_{\gamma}\in C.$
\item $C$ is a \emph{club}\index{club!-- of $[A]^{<\kappa}$} if it is closed and unbounded. 
\item $C$ is \emph{stationary}\index{stationary!-- subset of $[A]^{<\kappa}$} if $S$ has non-empty intersection with every club of $[A]^{<\kappa}.$  
\end{enumerate}  
\end{definition}

We will often use the following lemma. 

\begin{lemma}(Pressing Down Lemma)\index{Lemma!Pressing Down --} If $f$ is a regressive function on a stationary set $S\subseteq [A]^{<\kappa}$ (i.e. $f(x)\in x,$ for every non empty $x\in S$), then there exists a stationary set $T\subseteq S$ such that $f$ is constant on $T.$
\end{lemma}

For a proof of that lemma see \cite[Theorem 8.24]{Jechbook}.\\ 

Given a forcing $\mathbb{P}$ and conditions $p,q\in \mathbb{P},$ we use $p\leq q$ in the sense that $p$ is stronger than $q.$ 
Assume that $\mathbb{P}$ is a forcing notion in a model $V,$ we will use $V^{\mathbb{P}}$ to denote the class of $\mathbb{P}$-names. If $G\subseteq \mathbb{P}$ is a generic filter over $V,$ then $V[G]$ denotes the generic extension of $V$ determined by $G.$ If $a\in V^{\mathbb{P}}$ and $G\subseteq \mathbb{P}$ is generic over $V,$ then 
$a^G$ denotes the interpretation of $a$ in $V[G].$ 
Every element $x$ of the ground model $V$ is represented in a canonical way by a name $\check{x}.$ However, to simplify the notation, we will use just $x$ instead of $\check{x}$ in forcing formulas.

\begin{definition} Given a forcing $\mathbb{P}$ and a cardinal $\kappa,$ we say that 
\begin{enumerate}
\item $\mathbb{P}$ is \emph{$\kappa$-closed}\index{$\kappa$-closed} if and only if every decreasing sequence of conditions of $\mathbb{P}$ of size less than $\kappa$ has an infimum; 
\item $\mathbb{P}$ is \emph{$\kappa$-c.c.}\index{$\kappa$-c.c.} when every 
antichain of $\mathbb{P}$ has size less than $\kappa;$ 
\item $\mathbb{P}$ has the \emph{$\kappa$-covering property} if $\mathbb{P}$ preserves $\kappa$ as a cardinal and for every filter $G\subseteq \mathbb{P}$ generic over $V,$ every set $X\subseteq V$ in $V[G]$ of cardinality less than $\kappa$ is contained in a set 
$Y\in V$ of cardinality less than $\kappa$ in $V.$
\end{enumerate}
\end{definition}

We will use the following forcing notions.

\begin{definition}(The L\'evy Collapse) Let $\kappa<\lambda$ be two cardinals with $\kappa$ regular, 
\begin{enumerate}
\item we denote by $\Coll(\kappa, \lambda)$ the set $\{p: \kappa\to \lambda;\ \vert \dom(p)\vert<\kappa  \}$
ordered by reverse inclusion; 
\item if $\lambda$ is inaccessible, then $\Coll(\kappa, <\lambda):= \Pi_{\alpha<\lambda} \Coll(\kappa, \alpha).$ 
\end{enumerate}
\end{definition}

\begin{lemma}(L\'evy) Let $\kappa<\lambda$ be two cardinals with $\kappa$ regular, then 
$\Coll(\kappa, \lambda)$ collapses $\lambda$ onto $\kappa,$ i.e. $\lambda$ has cardinality $\kappa$ in the generic extension. Moreover, 
\begin{enumerate}
\item every cardinal $\alpha\leq \kappa$ in $V$ remains a cardinal in $V[G];$
\item if $\lambda^{<\kappa}= \lambda,$ then every cardinal $\alpha>\lambda$ remains a cardinal in the extension. 
\end{enumerate}
\end{lemma}

For a proof of that lemma see for example \cite[Lemma 15.21]{Jech}.

\begin{lemma}(L\'evy) If $\kappa$ is regular and $\lambda>\kappa$ is inaccessible. Then 
for every $G\subseteq \Coll(\kappa, <\lambda)$ generic over $V,$ 
\begin{enumerate}
\item every $\alpha$ such that $\kappa\leq \alpha<\lambda$ has cardinality $\kappa$ in $V[G];$
\item every cardinal $\leq \kappa$ and every cardinal $\geq \lambda$ remains a cardinal in $V[G].$
\end{enumerate}
Hence $V[G]\models \lambda= \kappa^+.$
\end{lemma}

For a proof of that lemma see for example \cite[Theorem 15.22]{Jech}.

We will assume familiarity with the theory of large cardinals and elementary embeddings, as developed for example in \cite{Kanamori}. We recall the definition of strongly compact and supercompact cardinals. 

\begin{definition} Let $\kappa$ be a regular uncountable cardinal, 
\begin{enumerate}
\item $\kappa$ is \emph{strongly compact} if and only if, every $\kappa$-complete filter on a set $S$ can be extended to a $\kappa$-complete ultrafilter on $S;$
\item $\kappa$ is \emph{supercompact} if and only if, for every cardinal $\lambda\geq \kappa,$ there exists an elementary embedding $j: V\to M$ with critical point $\kappa$ such that $j(\kappa)>\lambda$ and $M$ is closed by subsets of size $\lambda.$
\end{enumerate}
\end{definition}

The following two theorems will be deeply used in this paper. 

\begin{lemma} (Laver) \cite{Laver} If $\kappa$ is a supercompact cardinal, then there exists $L: \kappa \to V_{\kappa}$ such that: for all $\lambda,$ for all $x\in H_{\lambda^+},$ there is an elementary embedding $j: V\to M$ with critical point $\kappa$
such that $j(\kappa)>\lambda,$ ${}^\lambda M\subseteq M$ and $j(L)(\kappa)= x.$
\end{lemma}

\begin{lemma} (Silver) Let $j: M\to N$ be an elementary embedding between inner models of {\rm ZFC}. Let $\mathbb{P}\in M$ be a forcing and suppose that $G$ is $\mathbb{P}$-generic over $M,$ $H$ is $j(\mathbb{P})$-generic over $N,$ and $j[G]\subseteq H.$ Then there is a unique $j^*: M[G]\to N[H]$ such that $j^*\restr M= j$ and $j^*(G)= H.$ 
\end{lemma}

\begin{proof} If $j[G]\subseteq H,$ then the map $j^*(\dot{x}^{G})= j(\dot{x})^{H}$ is well defined and satisfies the required properties. \end{proof}

\section{The Strong and the Super Tree Properties}\label{sec:maindef}

In this section we introduce the strong and super tree properties. Although the main results presented in this paper do not concern the super tree property (just the strong tree property), for the sake of completeness we include the definition of this property as well. In order to define the strong and the super tree properties, we need to introduce the notion of 
$(\kappa,\lambda)$-tree\footnote{In Weiss Phd-thesis \cite{WeissPhd} $(\kappa, \lambda)$-trees were called \emph{$\mathscr{P}_{\kappa}\lambda$-thin lists}.}. 

\begin{definition}\label{main definition} Given a regular cardinal $\kappa\geq \omega_2$ and an ordinal $\lambda\geq \kappa,$ a \emph{$(\kappa, \lambda)$-tree}\index{$(\kappa, \lambda)$-tree} is a set $F$ satisfying the following properties:  
\begin{enumerate}
\item for every $f\in F,$ $f: X\to 2,$ for some $X\in [\lambda]^{<\kappa}$
\item for all $f\in F,$ if $X\subseteq \dom(f),$ then $f\restr X\in F;$
\item the set $\lev_X(F):= \{f\in F;\spazio \dom(f)=X \}$ is non empty, for all $X\in [\lambda]^{<\kappa};$
\item $\vert \lev_X(F) \vert<\kappa ,$ for all $X\in [\lambda]^{<\kappa}.$
\end{enumerate}
\end{definition}

The elements of a $(\kappa, \lambda)$-tree are called \emph{nodes}. Note that, despite the name, 
a $(\kappa, \lambda)$-tree is not a tree. In fact for a given node $f$ on some level $\lev_X,$ the set of all its \emph{predecessors} is $\{f\restr Y;\ Y\subseteq X\}$ and it is not well ordered. So the main difference between a $\kappa$-tree and a $(\kappa, \lambda)$-tree is that in the former the levels are indexed by ordinals which are well ordered, while in the latter we have a level for every set in $[\lambda]^{<\kappa}$ which is not even linearly ordered. As usual, when there is no ambiguity, we will simply write $\lev_X$ instead of $\lev_X(F).$ 

\begin{definition}\label{branches} Given a regular $\kappa\geq \omega_2,$ an ordinal $\lambda\geq \kappa$ and a $(\kappa, \lambda)$-tree $F,$
\begin{enumerate}
\item  a \emph{cofinal branch}\index{branch! cofinal -- of a $(\kappa, \lambda)$-tree} for $F$ is a function $b: \lambda \to 2$ such that $b\restr X\in \lev_X(F),$ for all $X\in[\lambda]^{<\kappa};$
\item an \emph{$F$-level sequence}\index{$F$-level sequence}\index{level sequence} is a function $D: [\lambda]^{<\kappa}\to F$ such that for every $X\in [\lambda]^{<\kappa},$ $D(X)\in \lev_X(F);$
\item given an $F$-level sequence $D,$ an \emph{ineffable branch} for $D$ is a cofinal branch $b: \lambda \to 2$ such that
$\{X\in [\lambda]^{<\kappa};\spazio b\restr X= D(X) \}$ is stationary. 
\end{enumerate}
\end{definition}

\begin{definition}\label{def: TP ITP} Given a regular cardinal $\kappa\geq \omega_2$ and an ordinal $\lambda\geq \kappa,$ 
\begin{enumerate}
\item $(\kappa, \lambda)$-\TP \index{$(\kappa, \lambda)$-TP} holds if every $(\kappa, \lambda)$-tree has a cofinal branch;
\item $(\kappa, \lambda)$-\ITP \index{$(\kappa, \lambda)$-ITP} holds if for every $(\kappa, \lambda)$-tree $F$ and for every $F$-level sequence $D,$ there is an an ineffable branch for $D;$
\item we say that $\kappa$ satisfies the \emph{strong tree property}\index{strong tree property} if $(\kappa, \mu)$-\TP holds, for all $\mu\geq \kappa;$
\item we say that $\kappa$ satisfies the \emph{super tree property}\index{super tree property} if 
$(\kappa,\mu)$-\ITP holds, for all $\mu\geq \kappa;$
\end{enumerate}
\end{definition}

We prove a simple result that will be used repeatedly. 

\begin{lemma}\label{simple lemma} Let $\kappa$ be a regular cardinal and $\lambda\geq \kappa.$ For every $\lambda^* >\lambda,$ every $(\kappa, \lambda)$-tree with no cofinal branches, can be extended to a 
$(\kappa, \lambda^*)$-tree with no cofinal branches. \end{lemma}

\begin{proof} Let $F$ be a $(\kappa, \lambda)$-tree with no cofinal branches and let $\lambda^*>\lambda.$ We define a $(\kappa, \lambda^*)$-tree $F^*$ as follows: for every $X\in [\lambda^*]^{<\kappa},$ we let 
\begin{center}
$f: X\to 2 \in F^* \Longleftrightarrow_{def}  \ f\restr (X\cap \lambda) \in F$ and for every 
$\alpha\in X\setminus \lambda,$ $f(\alpha)= 0.$ 
\end{center}
It is clear that $F^*$ extends $F,$ i.e. for every $X\in [\lambda]^{<\kappa},$ $\lev_X(F)= \lev_X(F^*).$ 
 We check that $F^*$ is a $(\kappa, \lambda)$-tree. Conditions $1$ and $3$ of Definition \ref{main definition} are trivially satisfied. Condition $2$ is easily proved: if $f: X\to 2$ is in $F^*$ and $Y\subseteq X,$ then by definition $f\restr (X\cap \lambda)\in F,$ hence $f\restr (Y\cap \lambda)\in F.$ Moreover, for every
 $\alpha\in Y\setminus \lambda,$ we have $f\restr Y(\alpha)= f(\alpha)=0.$ Therefore $f\restr Y\in F^*.$ It remains to prove that for every $X\in [\lambda^*]^{<\kappa},$ the level $\lev_X(F^*)$ has size less than $\kappa,$ but the function 
 $f\mapsto f\restr \lambda$ defines a bijection of $\lev_X(F^*)$ into $\lev_{(X\cap \lambda)}(F),$ so $\lev_X(F^*)$ has size less than $\kappa.$ If $F^*$ has a cofinal branch $b^*: \lambda^*\to 2,$ then $b^*\restr \lambda$ is a cofinal branch for $F$ as well,  because for every $X\in [\lambda]^{<\kappa},$ $b^*\restr X\in \lev_X(F^*)= \lev_X(F).$ Since $F$ has no cofinal branches, $F^*$ has no cofinal branches as required. \end{proof}

\section{The Strong Tree Property at Successors of Singular Cardinals}\label{sec: successors}

To prove the consistency of the usual tree property at $\aleph_{\omega+1},$ Magidor and Shelah first proved a more general result concerning the tree property at successors of singular cardinals. 

\begin{theorem} (Magidor and Shelah \cite{MagidorShelah}) Assume $\nu$ is a singular limit of strongly compact cardinals, then 
$\nu^+$ has the tree property. 
\end{theorem}

In this section we prove that under the same assumptions, even the \emph{strong} tree property is satisfied at 
$\nu^+.$ The structure of the proof is very closed to Magidor and Shelah's proof of the previous theorem, although we will prove that to get the strong tree property at $\nu^+,$ it is enough for $\nu$ to be a singular limit of cardinals satisfying a nice partition property. This result is very important in the following, since the proof of the consistency of the strong tree property at $\aleph_{\omega+1},$ will mimic the proof of this theorem.    

\begin{notation}
Let $\mu$ be a regular cardinal and let $\lambda\geq \mu$ be any ordinal. 
For every cofinal set $I\subseteq [\lambda]^{<\mu}$ we denote by $[[\ I\ ]]^2$ the set of all pairs 
$(X,Y)\in I\times I$ such that $X\subseteq Y.$
\end{notation}

\begin{definition} Let $\mu>\kappa$ be two regular cardinals and let $S\subseteq [\lambda]^{<\mu}$ be a cofinal set and $c: [[\ S\ ]]^2\to \gamma$ a function such that $\gamma<\kappa.$ We say that a cofinal set $H\subseteq S$ is a \emph{quasi homogenous} set of color $i<\gamma$ iff for every $X,Y\in H$ there is $W\supseteq X,Y$ in $H$ such that $c(X,W)=i= c(Y,W).$\end{definition}

\begin{definition} Given two regular cardinals $\nu\geq \kappa,$ 
we say that the principle $\varphi(\kappa, \nu)$ holds when for every $\lambda\geq \nu$ if 
$S\subseteq [\lambda]^{<\nu}$ is a stationary set, then every function 
$c: [[\ S\ ]]^2\to \gamma$ with $\gamma<\kappa$ has a quasi homogenous set $H$ which is also stationary.  
\end{definition}

We now prove that strongly compact cardinals satisfy $\varphi$ everywhere. 

\begin{theorem}\label{coloring} Let $\kappa$ be a strongly compact cardinal, then 
$\varphi(\kappa, \nu)$ holds for every regular $\nu\geq \kappa.$
\end{theorem}

\begin{proof}  Fix $\lambda\geq \nu$ and a function 
$c: [[\ S\ ]]^2\to \gamma$ where $\gamma<\kappa,$ and let $S\subseteq [\lambda]^{<\nu}$ be a stationary set. 
Consider all the sets of the form $C\cap S$ where $C\subseteq [\lambda]^{<\nu}$ is a club; they form a 
$\kappa$-complete family. Since $\kappa$ is strongly compact, there exists a $\kappa$-complete ultrafilter 
$U$ that contains all these sets. Note that every set in $U$ is stationary. In fact if $H\in U$ and $C$ is a club, then by definition $C\cap S\in U,$ hence $H\cap C\cap S$ is in U as well, and it is non empty. First we show that for every $X\in S,$ there is $i_X<\gamma$ and a set 
$H_X\subseteq S$ in $U$ such that for every $Y$ in $H_X$ we have $c(X,Y)= i_X.$ 
Assume for a contradiction that for every $i<\gamma,$ the set $K_i:= \{Y\in S;\ Y\supseteq X \textrm{ and }\ c(X,Y)\neq i \}\in U$ then, by the $\kappa$-completeness of $U,$ the intersection $\bigcap_{i<\gamma}C_i$ is in $U$ and it is empty, a contradiction. 
A similar argument proves that the function $X\mapsto i_X$ is constant on a set $H\in U;$ let $i$ be such that $i= i_X,$ for every $X\in H.$ Now, it is easy to see that $H$ is quasi-homogenous of color $i.$ Indeed, if $X,Y\in H,$ then $H_X\cap H_Y\cap H$ belongs to $U$ and it is, therefore, non empty. Let $Z\in H_X\cap H_Y\cap H,$ then we have $c(X,Z)= i= c(Y,Z)$ as required. \end{proof} 

\begin{theorem}\label{phi} Let $\nu$ be a singular cardinal such that $\nu= \lim_{i<cof(\nu)} \kappa_i$ where every $\kappa_i$ is an uncountable cardinal satisfying $\varphi(\kappa_i, \nu^+).$ Then $\nu^+$ has the strong tree property. 
\end{theorem}


\begin{proof} To simplify the notation we will assume that $\nu$ has countable cofinality, so $\nu= \lim_{n<\omega} \kappa_n.$ 
Suppose without loss of generality that $\langle \kappa_n\rangle_{n<\omega}$ is increasing. 
Let $\mu\geq \nu^+$ and let $F$ be a 
$(\nu^+, \mu)$-tree. For every $X\in [\mu]^{<\nu^+},$ let $\{f_i^X\}_{i<\vert \lev_X(F)\vert}$ be an enumeration of 
$\lev_X(F).$ First we ``shrink" the tree as follows.  

\begin{lemma} 
There exists $n<\omega$ and a stationary set
 $S\subseteq [\mu]^{<\nu^+},$ such that for all $X,Y\in S,$ there are $\zeta, \eta<\kappa_n$ such that $f_{\zeta}^X \restr (X\cap Y) = f_{\eta}^Y\restr (X\cap Y).$\end{lemma}

\begin{proof} Given a function $f\in Lev_X,$ we write $\#f = i$ for $i<\nu,$ when $f= f_i^X.$ 
Define $c: [[\ [\mu]^{<\nu^+} ]]^2\to \omega $ by $c(X,Y)=\min \{i;\ \#(f_0^Y\restr X)<\kappa_i \}.$ 
By hypothesis $\varphi(\kappa_0, \nu^+)$ holds, hence there is a stationary quasi homogenous set $S\subseteq [\mu]^{<\nu^+}$ of color $n<\omega.$ Then, for every $X,Y\in S,$ there is $Z\supseteq X,Y$ in $S$ such that $c(X,Z)= n= c(Y,Z).$ This means that 
$\#(f_0^Z\restr X),\ \#(f_0^Z\restr Y)<\kappa_n,$ namely there are $\zeta,\eta<\kappa_n$ such that $f_0^Z\restr X= f_\zeta^X$ and $f_0^Z\restr Y= f_\eta^Y.$ So we have 
$$f_\zeta^X \restr (X\cap Y) = f_0^Z\restr (X\cap Y)= f_\eta^Y\restr (X\cap Y),$$ as required. That completes the proof of the lemma. \end{proof}

Let $n$ and $S$ be as above, we prove the following fact. 

\begin{lemma} There is a cofinal $S'\subseteq S$ and an ordinal $\zeta<\kappa_n$ such that for all $X,Y\in S',$ we have
$f_{\zeta}^X \restr (X\cap Y)= f_{\zeta}^Y \restr (X\cap Y)$ (the set $S'$ is even stationary).
\end{lemma}

\begin{proof} For every $(X,Y)\in [[\ S\ ]]^2,$ we define $\bar{c}(X,Y)$ as the minimum couple $(\zeta,\eta)\in \kappa_n\times \kappa_n,$ in the lexicografical order, such that $f_\eta^Y\restr X= f_\zeta^X$ --- the function is well defined by definition of $n$ and $S.$ 
We can apply $\varphi(\kappa_{n+1}, \nu^+)$ to $\bar{c}$ as this can be seen as a function from $[[\ S\ ]]^2$ into $\kappa_n$ --- take any bijection $h: \kappa_n\times \kappa_n\to \kappa_n$ and apply $\varphi(\kappa_{n+1}, \nu^+)$ to $\bar{c} \circ h: [[\ S\ ]]^2 \to \kappa_n.$
So there exists a quasi homogenous stationary set $S'$ of color 
$(\zeta, \eta)\in \kappa_n\times \kappa_n,$ 
hence for every $X,Y\in S',$ there is $Z\supseteq X,Y$ in $S'$ such that 
$\bar{c}(X,Z)= (\zeta, \eta) =\bar{c}(Y,Z).$ This means that $f_\eta^Z\restr X= f_\zeta^X$ and $f_\eta^Z\restr Y= f_\zeta^Y.$ It follows that 
$$f_\zeta^X \restr (X\cap Y)= f_\eta^Z\restr (X\cap Y)= f_\zeta^Y\restr (X\cap Y).$$ That completes the proof of the lemma. \end{proof}

Now we conclude the proof of the theorem by defining a cofinal branch. Let $b:= \bigcup_{X\in S'} f_{\zeta}^X,$ by the previous lemma $b$ is a function. Moreover, for every $Y\in S'$ we have 
$$b\restr Y = \bigcup_{X\in S'} f_{\zeta}^X \restr Y= \bigcup_{X\in S'} f_{\zeta}^X \restr (X\cap Y)= \bigcup_{X\in S'} f_{\zeta}^Y\restr (X\cap Y)= f_{\zeta}^Y.$$
It follows that $b$ is a cofinal branch for $F.$ \end{proof}

\begin{coroll}\label{MS} Let $\nu$ be a singular limit of strongly compact cardinals, then $\nu^+$ has the strong tree property. 
\end{coroll}

\begin{proof} Apply Theorem \ref{phi} and Theorem \ref{coloring}. \end{proof}


Whether such result can be generalized to the \emph{super} tree property is still an open problem. 
Based on the analogy between supercompact cardinals and the super tree property, we can conjecture that the successor of a singular limit of \emph{supercompact} cardinals satisfy the super tree property. 

We conclude this section by proving the following fact. 

\begin{proposition} Given a regular cardinal $\kappa,$ if $\varphi(\kappa, \kappa)$ holds, then 
$\kappa$ has the strong tree property.
\end{proposition}

\begin{proof} Let $F$ be a $(\kappa, \lambda)$-tree where $\lambda\geq \kappa.$ 
For every $X\in [\lambda]^{<\kappa},$ let $\{f_i^X\}_{i<\gamma_X}$ be an enumeration of $\lev_X(F).$ 
We can assume without loss of generality that $\lambda^{<\kappa}$ is large enough so that 
$\gamma_X$ (the size of $\vert \lev_X(F)\vert$) is constant on a stationary set $S\subseteq [\lambda]^{<\kappa}$ --- indeed, if it is not the case we can take a larger $\lambda^*$ and use Lemma \ref{simple lemma}. So let $\gamma<\kappa$ be such 
that $\gamma_X= \gamma$ for every $X\in S.$ We define a function
$c: [[\ S\ ]]^2\to \gamma\times \gamma$ by letting 
$c(X,Y)$ be the the minimum couple $(i,j)$ in the lexicografical order such that $f_j^Y\restr X= f_i^X.$ The function $c$ can be seen as a function from 
$[[\ S\ ]]^2$ into $\gamma,$ so there exists a quasi-homogenous and stationary set 
$H\subseteq S.$ Assume $H$ has color $(i,j)\in \gamma\times \gamma,$ we let 
$b:= \bigcup_{X\in H} f_i^X$ and we prove that $b$ is a cofinal branch. 
Given $X,Y\in H,$ there is $Z\in H$ such that 
$X,Y\subseteq Z$ and $c(X,Z)= (i,j)= c(Y,Z).$ By definition of $c,$ we have 
\begin{enumerate}
\item $f_j^Z\restr X= f_i^X;$
\item $f_j^Z\restr Y= f_i^Y.$
\end{enumerate}

It follows that $f_i^X\restr (X\cap Y)= f_i^Y\restr (X\cap Y).$ Therefore $b$ is a function and for every $Y\in H,$ we have
$$b\restr Y= \bigcup_{X\in H} f_i^X \restr Y= \bigcup_{X\in H} f_i^X\restr (X\cap Y)= \bigcup_{X\in H} f_i^Y\restr (X\cap Y)= f_i^Y,$$ so $b$ is a cofinal branch. \end{proof}





\section{Systems}

To prove the consistency of the strong tree property at $\aleph_{\omega+1},$ we will work with a special structure that we call a ``system''. To understand this notion, suppose we are in the following situation. 
Let $\nu$ be a cardinal in a model $V$ and let $\mathbb{P}$ be a forcing notion with the $\nu^+$-covering property --- so that for every $\lambda\geq \nu^+,$ the set $([\lambda]^{<\nu^+})^V$ is cofinal in the $[\lambda]^{<\nu^+}$ of the generic extension. Assume $\dot{F}\in V^{\mathbb{P}}$ is a name for a 
$(\nu^+, \lambda)$-tree and for every $X\in [\lambda]^{<\nu^+},$ $\dot{e}_X$ is a $\mathbb{P}$-name for a
an enumeration of $\lev_X(\dot{F})$ (i.e. $\force_{\mathbb{P}}\ \dot{e}_X: \nu \to \lev_X(\dot{F})$ is onto).
For every $p\in \mathbb{P},$ we define a binary relation $S_p$ over the pairs $(X, \zeta),$ where 
$X\in [\lambda]^{<\nu^+}$ and $\zeta<\nu:$ 

\begin{center} $(X, \zeta)\ S_p\ (Y, \eta) \Longleftrightarrow_{def} \ p\force\ \dot{e}_X(\zeta)= \dot{e}_Y(\eta)\restr X.$ 
\end{center}

In other words, we have $(X, \zeta)\ S_p\ (Y, \eta)$ when $p$ forces that the $\eta$-th function on level $Y$ extends the $\zeta$-th function on level $X.$ The family $\{S_p\}_{p\in \mathbb{P}}$ satisfies the following definition. 

\begin{definition}\label{system} Given an ordinal $\lambda\geq \nu^+,$ a cofinal set $D\subseteq [\lambda]^{<\nu^+}$ and a family $\mathscr{S}:= \{S_i\}_{i\in I}$ of transitive, reflexive binary relations over $D\times \nu,$ we say that $\mathscr{S}$ is a \emph{system} if the following hold: 
\begin{enumerate}
\item if $(X, \zeta)\ S_i\ (Y,\eta)$ and $(X,\zeta)\neq (Y,\eta),$ then $X\subsetneq Y;$
\item for every $X\subseteq Y,$ if both $(X,\zeta)\ S_i\ (Z, \theta)$ and $(Y, \eta)\ S_i\ (Z, \theta),$\linebreak
then $(X, \zeta)\ S_i\ (Y, \eta);$
\item for every $X, Y\in D,$ there is $Z\supseteq X,Y$ and $\zeta_X, \zeta_Y, \eta\in \nu$ such that for some $i\in I$ we have $(X, \zeta_X)\ S_i\ (Z, \eta)$ and $(Y, \zeta_Y)\ S_i\ (Z,\eta)$ (in particular, if $X\subseteq Y,$ then $(X,\zeta_X)\ S_i\ (Y,\zeta_Y)$). 
\end{enumerate}
\end{definition}

To prove the consistency of the strong tree property at $\aleph_{\omega+1},$ we will have to deal with a system similar to the one defined above. In this section, we analyze some properties of these structures. 

The elements of $D\times \nu$ are called \emph{nodes of the system}. Given two nodes $u$ and $v,$ we say that they are \emph{$S_i$-incompatible,} for some $i \in I,$ if there is no $w\in D\times \nu$ such that 
$u\ S_i\  w$ and $v\ S_i\ w.$ We will say that a node $u$ belongs to a \emph{level $X$} if the first coordinate of $u$ is $X$ (i.e. $u=(X,\zeta),$ for some $\zeta\in \nu$). 

\

\begin{definition} Let $\{S_i\}_{i\in I}$ be a system on $D\times \nu$ and let $b: D\to \nu$ be a partial function.  
\begin{enumerate}
\item We say that $b$ is an \emph{$S_i$-branch} for some $i\in I,$ if the following holds. For every 
$X\in \dom(b)$ and for every $Y\in D$ such that $Y\subseteq X,$ we have
\begin{center}
$Y\in \dom(b)$ iff there exists $\zeta<\nu$ such that $(Y, \zeta)\ S_i\ (X, b(X)),$  
and $b(Y)$ is the unique $\zeta$ witnessing this.
\end{center}

\item We say that $b$ is a \emph{cofinal branch} for the system if it is an $S_i$-branch for some $i\in I,$ and $X\in \dom(X)$ for cofinally many $X$'s in $D.$  
\end{enumerate}
\end{definition}


We will often work with families of branches satisfying specific conditions. 

\begin{definition}\label{system of branches} Let $\{S_i\}_{i\in I}$ be a system on $D\times \nu,$ a \emph{system of branches} is a family 
$\{b_j\}_{j\in J}$ such that 
\begin{enumerate}
\item every $b_j$ is an $S_i$-branch for some $i\in I;$
\item for every $X\in D,$ there is $j\in J$ such that $X\in \dom(b_j).$
\end{enumerate}
\end{definition}

A lemma by Silver establishes that whenever we force with a forcing that has enough closure, it cannot add cofinal branches to a given tree. 

\begin{lemma} (Silver) Let $\tau, \kappa$ be regular cardinals, and suppose $\tau<\kappa	\leq 2^{\tau}.$ Let $\mathbb{P}$ be a $\tau^+$-closed forcing in a model $V$ and let $T$ be a $\kappa$-tree. 
Then for every generic extension $V[G]$ by $\mathbb{P},$ every branch of $T$ in $V[G]$ is in fact a member of $V.$ \end{lemma}

\begin{proof} We may assume that $\tau$ is minimal with $2^{\tau}\geq \kappa.$ 
Let $\dot{b}$ be a $\mathbb{P}$-name for a new branch. We build by induction for each $s\in {}^{\leq \tau+1} 2$ conditions $p_s$ and points $x_s$ of $T$ such that 
\begin{enumerate}
\item if $t \sqsubseteq s,$ then $p_s\leq p_t$ and $x_s>_T\ x_t;$
\item $p_s\force\ x_s \in \dot{b};$ 
\item for each $\alpha,$ the nodes $\{x_s;\ s\in {}^{\alpha}2 \}$ are all on the same level $\eta_{\alpha};$
\item for each $s\in {}^{<\tau} 2$ the nodes $x_{s\smallfrown 0}$ and $x_{s\smallfrown 1}$ are incompatible. 
\end{enumerate}
By minimality of $\tau,$ for every $\alpha<\tau,$ the set $\{x_s;\ s\in {}^{\alpha} 2 \}$ has size less than $\kappa,$ so we can choose $\eta_{\alpha+1}.$ The closure of $\mathbb{P}$ guarantees that the construction works at limit stages. In the end we have a contradiction, because the level $\eta_{\tau}$ must have fewer than $\kappa$ many nodes, yet we have constructed $2^{\tau}$ many distinct ones. \end{proof}

Now we want to generalize Silver's lemma to systems. More precisely, we are going to prove that if a $\kappa$-closed forcing adds a system of branches trough a ``small'' system, then a cofinal branch must already exist in the ground model (Theorem \ref{preservation theorem} below). Such a result generalizes a lemma by Sinapova (see \cite{Sinapova} Preservation Lemma) and will be used to prove the consistency of the strong tree property at $\aleph_{\omega+1}.$
First we prove the following lemma that provides a useful ``splitting argument''. 

\begin{lemma}\label{splitting} (Splitting Lemma) Let $\nu$ be a singular cardinal of countable cofinality and let $\lambda\geq \nu^+.$ Let $\{R_i\}_{i\in I}$ be a system on $D\times \tau$ (with 
$D\subseteq [\lambda]^{<\nu^+}$ cofinal) and let $\mathbb{P}$ be a forcing notion such that:
\begin{enumerate}
\item $\max(\vert I\vert, \tau )<\nu;$
\item $\mathbb{P}$ is $\kappa$-closed for some regular $\kappa$ 
between $\max(\vert I\vert, \tau )^+$ and $\nu;$
\item for some $p\in \mathbb{P},$ $\dot{b}\in V^{\mathbb{P}}$ and $i\in I,$ we have  
$$p\force \dot{b}\textrm{ is a cofinal $R_i$-branch}.$$
\end{enumerate}
If $V$ has no cofinal branches for the system, then for all $\eta<\kappa,$ we can find a sequence $\langle v_{\zeta};\ \zeta<\eta \rangle$ of pairwise $R$-incompatible elements of $D\times\tau$ such that for every $\zeta<\eta,$ there exists $q\leq p$ that forces $v_{\zeta}\in \dot{b}.$ 
\end{lemma}

\begin{proof} It might be helpful to point out that if $G$ is a generic filter containing $p,$ 
then in $V[G]$ the domain of $\dot{b}^G$ is a cofinal set in $([\lambda]^{<\nu^+})^V.$ We work in $V.$ Let $R:= R_i$ and let $E:= \{ u\in D\times \tau;\ \exists q\leq p( q\force u\in \dot{b})\}.$ First remark that, since $p$ forces that $\dot{b}$ is cofinal, the set $\{X\in D;\ \exists \zeta\in \tau\ (X,\zeta)\in E\}$ is cofinal. 
As $V$ has no cofinal branches for the system, we can find, for all $v\in E$ two $R$-incompatible nodes $w_1, w_2\in E$ such that $v\ R\ w_1,$ $v\ R\ w_2.$ 

We inductively define for all $\zeta< \eta$ two nodes $u_\zeta, v_\zeta\in E$ and a condition $p_\zeta\leq p$ such that:
\begin{enumerate}
\item $u_\zeta$ and $v_\zeta$ are $R$-incompatible; 
\item\label{claim:Rchain} for all $\varepsilon<\zeta,$ $u_{\varepsilon}\ R\ u_{\zeta}$ and $u_\varepsilon\ R\ v_\zeta;$
\item $p_{\zeta}\force u_{\zeta}\in \dot{b};$ 
\item the sequence $\langle p_\varepsilon;\ \varepsilon\leq \zeta  \rangle$ is decreasing;
\end{enumerate}

Let $u$ be any node in $E.$ From the remark above, there are $u_0,v_0\in E$ which are $R$-incompatible and both $u\ R\ u_0$ and 
$u\ R\ v_0$ hold. By definition of $E,$ there is a condition $p_0\leq p$ such that $p_0\force u_0\in \dot{b}.$

Let $\zeta>0$ and assume that $u_\varepsilon,$ $v_\varepsilon,$ $p_\varepsilon$ 
are defined for every $\varepsilon <\zeta.$ Let $q$ be stronger than every condition in $\{p_\varepsilon;\ \varepsilon<\zeta\}.$ By the inductive hypothesis (claim \ref{claim:Rchain}), the nodes 
$\langle u_\varepsilon;\ \varepsilon<\zeta  \rangle $ form an $R$-chain whose levels are sets 
in $[\lambda]^{<\nu^+}.$ The union of the levels of these nodes is a set $X$ in $[\lambda]^{<\nu^+}$ and since 
$\dot{b}$ is forced to be a cofinal $R$-branch we can find a node $h$ of level above $X$ and a condition $q^*\leq q$ such that $q^*\force h\in \dot{b}.$ It follows that $u_\varepsilon\ R\ h,$ for all $\varepsilon<\zeta.$ Since there is no cofinal branch in 
$V$ for the system, we can find two $R$-incompatible nodes $u_\zeta,v_\zeta\in E$ and a condition 
$p_\zeta\leq q^*$ such that $h\ R\ u_\zeta,$ $h\ R\ v_\zeta$ and $p_\zeta\force u_\zeta\in \dot{b}.$ That completes the construction.

The sequence $\langle v_\zeta;\ \zeta<\eta  \rangle$ is as required: for if $\zeta'<\zeta<\eta,$ then by definition $u_{\zeta'}$ and $v_{\zeta'}$ are $R$-incompatible, and 
$u_{\zeta'}\ R\ v_{\zeta},$ hence $v_{\zeta'}$ and $v_{\zeta}$ are $R$-incompatible as well. \end{proof}

\

\begin{theorem}\label{preservation theorem} (Preservation Theorem) In a model $V,$ we let $\nu$ be a singular cardinal of countable cofinality and let $\lambda\geq \nu^+.$ Let $\{R_i\}_{i\in I}$ be a system on $D\times \tau$ (with $D\subseteq [\lambda]^{<\nu^+}$ cofinal), let $\mathbb{P}$ be a forcing notion and let $G\subseteq \mathbb{P}$ a generic filter over $V.$ Assume that   
\begin{enumerate}
\item $\max(\vert I\vert, \tau )<\nu;$
\item $\mathbb{P}$ is $\kappa$-closed for some regular $\kappa$ between 
$\max(\vert I\vert, \tau )^+$ and $\nu;$
\item in $V[G]$ there is a system of branches $\{b_j\}_{j\in J}$ through $\{R_i\}_{i\in I}$ such that 
\begin{enumerate}
\item $J\in V$ and $\vert J\vert^+<\kappa;$ 
\item\label{cofinal} for some $j\in J,$ the branch $b_j$ is cofinal. 
\end{enumerate}
\end{enumerate}
Then, for some $i\in I,$ there exists in $V$ a cofinal $R_i$-branch.   
\end{theorem}

\begin{proof} Suppose for contradiction that $V$ has no cofinal branches for the system $\{R_i\}_{i\in I}.$ 
Let $\{\dot{b}_j\}_{j\in J}$ be $\mathbb{P}$-names for the branches of the system of branches in the generic extension. The idea of the proof is similar to the proof of Silver's lemma above and it follows three steps. 
\begin{enumerate}
\item We consider just the $\dot{b}_j$'s that are forced to be cofinal and for every such $\dot{b}_j,$ we use the Splitting Lemma to build $\eta$ many incompatible nodes that are forced to belong to $\dot{b}_j,$ where $\eta$ is a cardinal between $\max(\vert J\vert, \vert I\vert, \tau)$ 
and $\kappa.$ 
\item By using the $\kappa$-closure of $\mathbb{P}$ and the fact that there are less than $\kappa$ many possible cofinal branches, we find a name $\dot{b}$ for a $R$-branch and $\eta$ many 
$R$-incompatible nodes $\langle u_{\gamma};\ \gamma<\eta \rangle$ that are forced by ``nice conditions'' to belong to $\dot{b}.$ 
\item As $\eta<\nu^+,$ all these nodes are below some level $X\in D$ and we can find a node $w$ on a level above $X$ which is forced by those conditions to belong to $\dot{b}$ as well. Then we have a contradiction, as $w$ stands in the relation $R$ with $R$-incompatible nodes below it.
\end{enumerate}

We work in $V.$ 
Fix, for every $j\in J$ a condition $p_j$ deciding whether or not $\dot{b}_j$ is cofinal. We can choose the 
$p_j$'s so that they form a decreasing sequence, then by the $\kappa$-closure of $\mathbb{P}$ (recall $\vert J\vert<\kappa$) there exists a condition $p$ deciding, for every $j\in J$ whether or not $b_j$ is cofinal. 
We let $B:= \{j\in J;\ p\force \dot{b}_j\textrm{ is not cofinal} \}.$
For every $j\in B,$ 
fix $X_j\in [\lambda]^{<\nu^+}$ such that $p$ forces that $\dom(\dot{b}_j)$ has empty intersection with every $Y\supseteq X_j.$ Since $B$ has size less than $\nu,$ the set 
$X^*:= \underset{j\in B}{\bigcup} X_j$ is in $[\lambda]^{<\nu^+}.$ 
Let $C^*:=\{Z\in D;\ X^*\subseteq Z \}.$
Define $A:= \{j\in J;\ p\force \dot{b}_j \textrm{ is cofinal}   \},$
then by hypothesis $A$ is non empty (claim \ref{cofinal}). Moreover, by strengthening $p$ if necessary, we can 
assume 

\begin{equation} \label{eq:one} p\force \forall X\in C^*\exists j\in A( X\in \dom(\dot{b}_j)) 
\end{equation}

\

(use condition $2$ of Definition \ref{system of branches} and the definition of $C^*$). 
For every $a\in A,$ we denote by $R_a$ the relation in the system such that 
$p\force \dot{b}_a \textrm{ is an $R_a$-branch}.$
Fix a regular cardinal $\eta$ between 
$\max(\vert J\vert, \vert I\vert,  \tau )$ and $\kappa,$ we prove the following claim.  

\begin{claim} Let $\vartriangleleft$ be a well ordering (strict) of $A.$ For every $a\in A,$ we can define 
$\langle q_{\gamma}^a;\ \gamma<\eta \rangle $ and $\langle u_{\gamma}^a;\ \gamma<\eta \rangle$ such that 
\begin{enumerate}
\item for all $\gamma<\eta,$ $q_{\gamma}^a\leq p$ and $q_{\gamma}^a\force u_{\gamma}^a\in \dot{b}_a,$ 
\item the nodes $\langle u_{\gamma}^a;\ \gamma<\eta \rangle$ are pairwise $R_a$-incompatible, 
\item\label{claim:decreasing} for all $\gamma<\eta,$ the sequence $\langle q_{\gamma}^c;\ c\vartriangleleft a\rangle$ is decreasing, 
i.e. if $b \vartriangleleft c,$ then $q_{\gamma}^{c}\leq q_{\gamma}^{b}.$ 
\end{enumerate}   
\end{claim}

\begin{proof} We proceed by induction on the ordering $\vartriangleleft.$ Assume that the sequences have been defined up to $a\in A$ (i.e. for every $c\vartriangleleft a$).  For every $\gamma<\eta,$ let $r_{\gamma}$ be stronger than every condition in the set $\{q_{\gamma}^{c};\ c\vartriangleleft a \}$ (the sequence $\langle q_{\gamma}^{c};\ c\vartriangleleft a \rangle$ is decreasing by claim \ref{claim:decreasing}) and let $E_{\gamma}:= \{ u\in D\times \tau;\ \exists q\leq r_{\gamma}( q\force u\in \dot{b}_a)\}.$ For all $\gamma<\eta,$ there exists $\langle v_{\zeta}^{\gamma};\ \zeta<\eta \rangle$ like in the conclusion of Lemma \ref{splitting} applied to $r_{\gamma}$ and $\dot{b}_a.$ Let $X_{\gamma}\in [\lambda]^{<\nu^+}$ be such that the level of each $v_{\zeta}^{\gamma}$ is below $X_{\gamma}$ and let $X^*\supsetneq \underset{\gamma<\eta}{\bigcup} X_{\gamma}$ in $D.$ We want to define the sequence $\langle u_{\gamma}^a;\ \gamma<\eta \rangle$ with each $u_{\gamma}^a\in E_{\gamma}$ belonging to a level above $X^*.$ We proceed by induction: suppose 
we have defined $\langle u_{\gamma}^a;\ \gamma<\delta \rangle $ for some $\delta<\eta.$ For every $\gamma<\delta,$ there is at most one $\zeta<\eta$ such that 
$v_{\zeta}^{\delta}\ R_a\ u_{\gamma}^a$ 
(because the $v_{\zeta}^{\delta}$'s are pairwise $R_a$-incompatible), let $\zeta_{\gamma}$ be that unique index if it exists and let $\zeta_{\gamma}$ be $0$ otherwise. Choose 
$\zeta\in \eta\setminus \{\zeta_{\gamma}^{\delta};\ \gamma<\delta \},$ then for all $\gamma<\delta,$ the nodes $v_{\zeta}^{\delta}$ and $u_{\gamma}^a$ are $R_a$-incompatible. 
Let $u_\delta^a\in E_\delta$ be such that $v_{\zeta}^{\delta}\ R_a\ u_\delta^a.$ Then, for all $\gamma<\delta,$ the nodes $u_{\gamma}^a$ and $u_\delta^a$ are $R_a$-incompatible. Since for every $\gamma<\eta,$ we have $u_\gamma^a\in E_\gamma,$ we can find a condition $q_\gamma^a\leq r_{\gamma}$ such that $q_{\gamma}^a\force u_{\gamma}^a\in \dot{b}_a.$ That completes the proof of the claim.  \end{proof}

We return to the proof of the theorem. Condition \ref{claim:decreasing} above guarantees that 
for every $\gamma<\eta,$ the sequence $\langle q_{\gamma}^a;\ a\in A \rangle$ is decreasing. Since 
$A$ has size less than $\kappa,$ we can find for every $\gamma<\eta,$ a condition $p_{\gamma}$ stronger than all the conditions $\langle q_{\gamma}^a;\ a\in A\rangle$ and there is $Y_\gamma\in D$ such that the nodes in $\{u_\gamma^a;\ a\in A\} $ belong to levels below $Y_\gamma.$ Let $Y^*\in C^*$ be such that 
$Y^*\supseteq \underset{\gamma}{\bigcup}Y_{\gamma}.$ For all $\gamma<\eta,$ we fix $p_{\gamma}^*, w_{\gamma}$ and $a_{\gamma}$ such that $p_{\gamma}^*\leq p_{\gamma},$ $w_\gamma$ is a node on level $Y^*,$ $a_{\gamma}\in A$ and 
$p_{\gamma}^*\force w_{\gamma}\in \dot{b}_{a_{\gamma}}$ (use Equation \ref{eq:one}). Since $\vert A\vert, \tau<\eta,$ there is $w^*$ on level $Y^*$ and $a^*\in A$ such that $w_{\gamma}=w^*,$ $a_{\gamma}= a^*,$ for 
almost all $\gamma<\eta.$ Let $b^*:=\dot{b}_{a^*}.$ Given two distinct 
$\gamma,\delta<\eta$ large enough, if $u:= u_{\gamma}^{a^*}$ and 
$v:= u_{\delta}^{a^*},$ then the following hold:
\begin{enumerate}
\item $p_{\gamma}^*\force u\in b^*,$ $p_{\gamma}^*\force w^*\in b^*;$
\item $p_\delta^*\force v\in b^*,$ $p_\delta^*\force w^*\in b^*.$
\end{enumerate}

It follows that $u\ R_{a^*} w$ and $v\ R_{a^*} w.$ However, $u$ and $v$ are $R_{a^*}$-incompatible by definition,  
and that leads to a contradiction.  \end{proof}

\section{The Strong Tree Property at $\aleph_{\omega+1}$}

Now we are ready to prove the consistency of the strong tree property at $\aleph_{\omega+1}.$ The structure of the proof of this theorem is motivated by Neeman \cite{Neeman}. 

\begin{theorem} Let $\langle \kappa_n\rangle_n<\omega$ be an increasing sequence of indestructibly supercompact cardinals. There is a strong limit cardinal $\mu<\kappa_0$ of cofinality $\omega$ such that by forcing over $V$ with the poset 
$$\Coll(\omega, \mu)\times \Coll(\mu^+, <\kappa_0)\times \Pi_{n<\omega} \Coll(\kappa_n, <\kappa_{n+1}),$$ one gets a model where the strong tree property holds at $\aleph_{\omega+1}.$ 
\end{theorem}

\begin{proof} Let $\kappa$ denote $\kappa_0,$ for every $\mu<\kappa$ we let 

\begin{enumerate}
\item $\mathbb{R}(\mu):= \Coll(\omega, \mu)\times \Coll(\mu^+, <\kappa_0)\times \Pi_{n<\omega} \Coll(\kappa_n, <\kappa_{n+1}),$
\item $\mathbb{L}(\mu):= \Coll(\omega, \mu)\times \Coll(\mu^+, <\kappa_0),$
\item $\mathbb{C}:= \Pi_{n<\omega} \Coll(\kappa_n, <\kappa_{n+1}).$ 
\end{enumerate}

Assume that $\nu= \sup_n{\kappa_n}.$ Note that for every $\mu<\kappa,$ the forcing $\mathbb{R}(\mu)$ produces a model where $\aleph_{\omega+1}= \nu^+.$ Fix $H:= \Pi_{n<\omega} H_n\subseteq \mathbb{C}$ generic over $V.$ We work in $W:=V[H].$ Assume for a contradiction that in every extension of $W$ by $\mathbb{L}(\mu)$ with $\mu<\kappa$ strong limit of cofinality $\omega,$ the strong tree property fails at $\nu^+.$ For every such $\mu,$ let $\lambda_{\mu}$ and $\dot{F}(\mu)\in W^{\mathbb{L}(\mu)}$ be a name for a $(\nu^+, \lambda_{\mu})$-tree with no cofinal branches. Let $\lambda= \sup_{\mu<\kappa } \lambda_\mu,$ without loss of generality we can assume that $\lambda_{\mu}= \lambda$ for every $\mu,$ since a $(\nu^+, \lambda_{\mu})$-tree with no cofinal branches can be extended to a $(\nu^+, \lambda)$-tree with no cofinal branches (by Lemma \ref{simple lemma}). Note that for every $\mu$ the poset $\mathbb{L}(\mu)$ has the $\nu^+$-covering property since it is $\kappa_0$-c.c.. Therefore, $([\lambda]^{<\nu^+})^W$ is cofinal in the $[\lambda]^{<\nu^+}$ of any generic extension of $W$ by $\mathbb{L}(\mu).$
Given $X,Y\in [\lambda]^{<\nu^+}$ and $\zeta,\eta<\nu,$ we will write
$\force_{\mathbb{L}(\mu)} (X,\zeta) <_{\dot{F}_{\mu}} (Y, \eta)$
when
$$\force_{\mathbb{L}(\mu)} \textrm{ the $\eta$'th function on level $Y$ extends the
$\zeta$'th function on level $X$}$$

(i.e. for every $\mu$ and $X,$ we fix an $\mathbb{L}(\mu)$-name $\dot{e}_X^{\mu}$ for an enumeration of the level of $X$ into at most $\nu$ elements, then we write 
$\force_{\mathbb{L}(\mu)} (X,\zeta) <_{\dot{F}_{\mu}} (Y, \eta)$
when $\force_{\mathbb{L}(\mu)} \dot{e}_X^{\mu}(\zeta)=\dot{e}_Y^{\mu}(\eta)\restr X$). Consider the following set 
 $$I:= \{(a,b,\mu);\ \mu<\kappa \textrm{ is strong limit of cof $\omega$ and }(a,b)\in \mathbb{L}(\mu)\}.$$ We define a system $\mathscr{S}= \{S_i\}_{i\in I}$ on $[\lambda]^{<\nu^+}\times \nu$ as follows. Given $i= (a,b,\mu)\in I,$ for every $X,Y\in [\lambda]^{<\nu^+}$ and for every $\zeta, \eta<\nu,$ we let   
$$(X, \zeta)\ S_i\ (Y,\eta) \Longleftrightarrow_{def} (a,b)\force (X,\zeta)<_{\dot{F}_\mu} (Y,\eta).$$ 






First we prove that we can shrink the system. 

\begin{lemma}\label{system lemma} There is, in $W,$ an integer $n<\omega$ and a cofinal set $D\subseteq [\lambda]^{<\nu^+}$ such that $\{\ S_i\restr D\times \kappa_n\}_{i\in I}$ is a system.  
\end{lemma}

\begin{proof} $\kappa$ is indestructible supercompact, so we can fix $j: W\to W^*$ a $\sigma$-supercompact elementary embedding with critical point $\kappa,$ where $\sigma$ is large enough for the argument that follows. We have $a^*:=j[\lambda]\in W^*\cap [j(\lambda)]^{<j(\nu^+)}.$ 
Let $F^*$ be the name $j(\dot{F})(\nu),$ where $\dot{F}$ is the map $\mu \mapsto \dot{F}(\mu).$ 
We denote by $\ll \lambda \gg^{<\nu^+}$ the set of all the strictly increasing sequences from an ordinal 
$\alpha<\nu^+.$ into $\lambda.$ For every $s\in\ \ll \lambda \gg^{<\nu^+},$ the image of $s$ 
is a subset of $[\lambda]^{<\nu^+}.$  
We define a sequence $\langle (p_s, q_s, \zeta_s, n_s);\ s\in\ \ll \lambda\gg^{<\nu^+} \rangle $ such that

\begin{enumerate}
\item $(p_s, q_s)\in \Coll(\omega, \nu)\times \Coll(\nu^+, <j(\kappa)),$ $n_s<\omega,$ and $\zeta_s<j(\kappa_{n_s});$ 
\item $(p_s, q_s)\force (j[Im(s)], \zeta_s)<_{F^*} (a^*, 0);$
\item for every $t\sqsubseteq s$ in $\ll \lambda\gg^{<\nu^+},$ we have $q_s\leq q_t.$ 
\end{enumerate}

The sequence is inductively defined as follows.  
Let $s: \alpha\to \lambda$ be a strictly increasing sequence, assume by inductive hypothesis that 
$$\langle (p_s, q_s, \zeta_s, n_s);\ s\in\ \ll \lambda\gg^{<\alpha} \rangle $$ is defined. By condition $(3),$ the sequence $\langle q_{s\restr \beta};\ \beta<\alpha\rangle$ is decreasing. Moreover, $\Coll(\nu^+, <j(\kappa))$ is $\nu^+$-closed, so there exists a lower bound $\bar{q}_s$ for $\langle q_{s\restr \beta};\ \beta<\alpha\rangle.$  The set $j[Im(s)]$ is in $[j(\lambda)]^{<j(\nu^+)},$ so there exists $p_s\in \Coll(\omega, \nu),$ $q_s\leq \bar{q}_s$ in $\Coll(\nu^+, <j(\kappa))$ and $\zeta_s<j(\nu)$ such that $$(p_s, q_s)\force (j[Im(s)], \zeta_s)<_{F^*} (a^*, 0).$$ 
If we let $n_s$ be the minimum integer such that $\zeta_s<j(\kappa_{n_s}),$ then $p_s, q_s,\zeta_s$ and $n_s$ satisfy conditions $1,$ $2$ and $3$ for the sequence $s.$ That completes the definition.

For every $X\in [\lambda]^{<\nu^+}$ we denote by $s_X$ the unique strictly increasing sequence whose 
image is $X$ (i.e. $s_X: o.t. (X) \to \lambda$ and $Im(s_X):= X$). As $\Coll(\omega, \nu)$ has size less than $\lambda^{<\nu^+},$ there is a condition $p$ and a cofinal set $D\subseteq [\lambda]^{<\nu^+}$ such that for every $X\in D,$ we have 
$p= p_{s_{X}}.$ By shrinking $D,$ we can also assume that there exists $n<\omega$ such that $n= n_{s_{X}},$ for every $X\in D.$ 

\begin{claim} $\{\ S_i\restr D\times \kappa_n\}_{i\in I}$ is a system. 
\end{claim}

\begin{proof} We just have to prove that it satisfies condition $(3)$ of Definition \ref{system}. Fix $X,Y\in D,$ by construction we have
 
\begin{enumerate}
\item $(p, q_{s_X})\force (j[X], \zeta_X)<_{F^*} (a^*, 0),$
\item $(p, q_{s_Y})\force (j[Y], \zeta_Y)<_{F^*} (a^*, 0).$
\end{enumerate}

Take any set $Z$ in $D$ such that $s_Z\sqsupseteq s_X, s_Y$ (in particular $Z\supseteq X,Y$), then 
$q_Z$ is stronger than both $q_X$ and $q_Y.$ 
Therefore, the condition $(p, q_Z)$ forces that: 

\begin{enumerate}
\item[(i)] $(j[X], \zeta_X)<_{F^*} (a^*, 0);$
\item[(ii)] $(j[Z], \zeta_Z)<_{F^*} (a^*, 0);$
\item[(iii)] $(j[Z], \zeta_Z)<_{F^*} (a^*, 0);$
\item[(iv)] $(j[Y], \zeta_Y) <_{F^*} (a^*, 0).$
\end{enumerate}  

From $(i)$ and $(ii)$ follows $(p,q_Z)\force (j[X], \zeta_X)<_{F^*} (j[Z], \zeta_Z);$ from $(iii)$ and $(iv)$ follows $(p,q_Z)\force (j[Y], \zeta_Y)<_{F^*} (j[Z], \zeta_Z).$ Then, by elementarity, there exists $\mu<\kappa$ and $(\bar{p}, \bar{q})\in \mathbb{L}(\mu)$ and $\bar{\zeta}_X, \bar{\zeta}_Y, \bar{\zeta}_Z<\kappa_n$ such that 

$$(\bar{p}, \bar{q})\force (X, \bar{\zeta}_X)<_{\dot{F_\mu}} (Z, \bar{\zeta}_Z)\textrm{ and } 
(Y, \bar{\zeta}_Y)<_{\dot{F}_\mu} (Z, \bar{\zeta}_Z).$$
 
If $i= (\bar{p}, \bar{q}, \mu),$ then we just proved 
$(X, \bar{\zeta}_X)\ S_i\ (Z, \bar{\zeta}_Z)$ and $(Y, \bar{\zeta}_Y)\ S_i\ (Z, \bar{\zeta}_Z).$ \end{proof}  

That completes the proof of the lemma. \end{proof}

To simplify the notation, we define $R_i:= S_i\restr D\times \kappa_n,$ for every $i\in I.$ 

Let $m= n+2,$ by the indestructibility of $\kappa_{m+1}$ forcing over $W= V[H]$ with $\Coll(\kappa_m, \gamma)^V$ for sufficiently large $\gamma,$ adds an elementary embedding $\pi: V[H]\to M[H^*]$ with critical point $\kappa_{m+1}$ and $\pi(\kappa_m)>\sup\pi[\lambda]$ (use standard arguments for extending embeddings).   

\begin{lemma} There is in $V[H^*]$ a system of branches $\{b_j\}_{j\in J}$ for the system $\{ R_i\}_{i\in I}$ with $J= I\times \kappa_n,$ such that for some $j\in J,$ the branch $b_j$ is cofinal.  
\end{lemma}

\begin{proof} First note that since $\kappa_n, \vert I\vert <cr(\pi),$ we may assume that $\pi(I)=I$ and 
$\pi ( \{ R_i \}_{i\in I} )= \{\pi(R_i)\}_{i\in I}.$ This is a system on $\pi(D)\times \kappa_n.$ Let $a^*$ be a set in 
$\pi(D)$ such that $\pi[\lambda]\subseteq a^*.$ For every $(i, \delta)\in I\times \kappa_n,$ let $b_{i,\delta}$ be the partial map sending each $X\in D$ to the unique $\zeta<\kappa_n$ such that $(\pi[X], \zeta)\ \pi(R_i)\ (a^*, \delta)$ if such $\zeta$ exists. By elementarity, every $b_{i,\delta}$ is an $R_i$-branch. Condition $(2)$ of Definition \ref{system of branches} is satisfied as well: indeed, if $X\in D,$ then by condition $(3)$ of Definition \ref{system}, there exists $\zeta,\eta<\kappa_n$ and $i\in I$ such that $(\pi[X], \zeta)\ \pi(R_i)\ (a^*, \eta),$ hence $X\in \dom(b_{i, \eta}).$ It remains to prove that for some $j\in J,$ $b_j$ is cofinal. For every $X\in D,$ we fix 
$i_X,\delta_X$ such that $X\in \dom(b_{i_X, \delta_X}).$ The set $I$ has size less than $\kappa_m$ in $W,$ moreover 
$\Coll(\kappa_m, \gamma)^V$ 
is $\kappa_m$-closed in $V[H_m\times H_{m+1}\times \dots]$ and $W= V[H]$ is a $\kappa_m$-c.c. forcing extension of $V[H_m\times H_{m+1}\times \dots],$ so $I$ has size $<\kappa_m$ even in $V[H^*].$ On the other hand $\vert D\vert\geq \kappa_m,$ so there exists a cofinal $D'\subseteq D$ and $i,\delta$ in $V[H^*]$ such that $i= i_X$ and $\delta= \delta_X,$ for every $X\in D'.$ This means that $X\in \dom(b_{i,\delta})$ for every $X\in D',$ namely $b_{i,\delta}$ is a cofinal branch. \end{proof}

$V[H^*]$ is a $\kappa_m$-closed forcing extension of $V[H]= W,$ so we can apply the Preservation Theorem. 
Therefore a cofinal $R_i$-branch $b$ exists in $W,$ for some $i\in I.$ 
Assume that $i= (a,b,\mu),$ for every $X\subseteq Y$ in $\dom(b),$ we have 
$$(a,c)\force (X, b(X))<_{\dot{F}_{\mu}} (Y, b(Y)).$$ 
If $G_0\times G_1\subseteq \mathbb{L}(\mu)$ is any generic filter containing the condition $(a,c),$ then the branch $b$ determines a cofinal branch for $\dot{F}_{\mu}^{G_0\times G_1}$ in $W[G_0\times G_1]$ 
contradicting the fact that $\dot{F}_{\mu}$ is a name for a $(\nu^+, \lambda)$-tree with no cofinal branches. This completes the proof of the theorem. \end{proof}

\section{Conclusions}

We proved that if infinitely many supercompact cardinals exist in a model $V,$ then there is a forcing extension of $V$ where $\aleph_{\omega+1}$ has the strong tree property. We do not know whether 
$\aleph_{\omega+1}$ can consistently satisfy even the \emph{super} tree property. 

We also know (see \cite{Fontanella2}) that from infinitely many supercompact cardinals, one can build a model where the super tree property (hence in particular the strong tree property) holds at every cardinal 
of the form $\aleph_{n+2},$ where $n<\omega.$ Then, it is natural to ask whether it is possible to combine the two consistency results and prove from infinitely many supercompact cardinals, the consistency of the strong tree property ``up to'' 
$\aleph_{\omega+1},$ i.e. at every regular cardinal $\leq \aleph_{\omega+1}$ (above $\aleph_1$). These problems remain open.



\bigskip

\end{document}